\documentclass[12pt]{article}
\usepackage{amsmath,amsthm,amsfonts,amssymb}
\usepackage[pdftex,pdfborder={0 0 0}]{hyperref}
\usepackage{fullpage}
\usepackage{multirow}
\usepackage{array}
\usepackage{bbm}
\usepackage[noblocks]{authblk}
\usepackage{verbatim}
\usepackage{tikz}
\usepackage{float}
\usepackage{enumerate}
\usepackage{diagbox}
\usepackage{comment}
\usetikzlibrary{arrows}
\usepackage{pifont}

\newtheorem{theorem}{Theorem}[section]
\newtheorem{lemma}[theorem]{Lemma}

\newtheorem{corollary}[theorem]{Corollary}

\theoremstyle{definition}
\newtheorem{definition}[theorem]{Definition}

\newtheorem{remark}[theorem]{Remark}

\newlength{\Oldarrayrulewidth}

\newcommand{\Z}{\mathbb{Z}}

\newcommand{\p}{\textup{\textbf{p}}}

\newcommand{\res}{\textup{res}}

\renewcommand{\mod}[2]{\equiv#1\textup{ (mod }#2\textup{)}}

\def\m@th{\mathsurround=0pt}
\def\sm#1{\null\,\vcenter{\baselineskip9pt\lineskip.23ex\m@th
    \ialign{\hfil$\scriptstyle##$\hfil&&\ \hfil$\scriptstyle##$\hfil\crcr
    \mathstrut\crcr\noalign{\kern-\baselineskip}
    #1\crcr\mathstrut\crcr\noalign{\kern-\baselineskip}}}\,}
\def\smnp#1{\null\,\vcenter{\baselineskip9pt\lineskip.23ex\m@th
    \ialign{\hfil$\scriptstyle##$\hfil&&\ \ \hfil$\scriptstyle##$\hfil\crcr
    \mathstrut\crcr\noalign{\kern-\baselineskip}
    #1\crcr\mathstrut\crcr\noalign{\kern-\baselineskip}}}\,}

\begin{document}

\title{On the Properties of Fibotomic Polynomials}

\author[1]{Cameron~Byer\thanks{chunecake@gmail.com}}

\author[2]{Tyler~Dvorachek\thanks{dvortj19@uwgb.edu}}

\author[3]{Emily~Eckard\thanks{e.m.eckard@email.msmary.edu}}

\author[4]{Joshua~Harrington\thanks{joshua.harrington@cedarcrest.edu}}

\author[5]{Lindsey~Wise\thanks{wiselm1@appstate.edu}}

\author[6]{Tony~W.~H.~Wong\thanks{wong@kutztown.edu}}

\affil[1]{Department of Mathematics, Eastern Mennonite University}
\affil[2]{Department of Mathematics, University of Wisconsin-Green Bay}
\affil[3]{Department of Mathematics, Mount St.\ Mary's University}
\affil[4]{Department of Mathematics, Cedar Crest College}
\affil[5]{Department of Mathematics, Appalachian State University}
\affil[6]{Department of Mathematics, Kutztown University of Pennsylvania}
\date{\today}

\maketitle

\begin{abstract}
Define the $n$-th fibotomic polynomial to be the product of the monic irredicible factors of the $n$-th Fibonacci polynomial which are not factors of any Fibonacci polynomial of smaller degree.  In this paper, we prove a number of properties of the fibotomic polynomials. This includes determining the discriminant of the fibotomic polynomials and the resultant of pairs of fibotomic polynomials. Furthermore, we completely determine the factorization form of the fibotomic polynomials in prime fields. Results are also generalized for the bivariate homogenous fibotomic polynomials.
\end{abstract}

\section{Introduction}\label{sec:intro}

The well-known Fibonacci polynomials are defined by letting $F_1(x)=1$, $F_2(x)=x$, and $F_n(x)=F_{n-1}(x)\cdot x+F_{n-2}(x)$ for all integers $n\geq 3$.  In 1969, Webb and Parberry \cite{wp} showed that $F_n(x)$ is irreducible in $\mathbb{Z}[x]$ if and only if $n$ is prime.  The following definition was first introduced by Levy \cite{levy} in 2001.

\begin{definition} 
Let $n\geq 2$ be an integer.  The $n$-th \emph{fibotomic polynomial}, written as $\Psi_n(x)$, is the product of the monic irreducible factors of $F_n(x)$ which are not factors of $F_k(x)$ for any $k<n$.  For consistency, we define $\Psi_1(x)=1$. Hence, 
$$F_n(x)=\prod_{d\mid n}\Psi_d(x)$$
for all positive integers $n$.
\end{definition}

It follows from Webb and Parberry's result that for any prime $p$, $\Psi_p(x)=F_p(x)$ is irreducible in $\mathbb{Z}[x]$.  It was further shown by Levy that $\Psi_n(x)$ is irreducible in $\mathbb{Z}[x]$ for every integer $n\geq 2$.  However, Kitayama and Shiomi \cite{ks} showed that $\Psi_n(x)$ is often reducible in finite fields. More recently, Sagan and Tirrell \cite{st} studied the bivariate Lucas polynomials and their factorization using Lucas atoms. The bivariate Lucas polynomials are defined such that $L_1(s,t)=1$, $L_2(s,t)=s$, and $L_n(s,t)=L_{n-1}(s,t)\cdot s+L_{n-2}(s,t)\cdot t$ for all integers $n\geq3$.

It seems that the bivariate Lucas polynomials are more general than the Fibonacci polynomials. However, a simple homogenization of the Fibonacci polynomials together with a substitution allows us to transform a Fibonacci polynomial back to a bivariate Lucas polynomial. Define $F_n(x,y)=y^{n-1}F_n\big(\frac{x}{y}\big)$ for all positive integers $n$. Then $F_n(x,y)$ is a homogeneous polynomial since the degree of $F_n(x)$ is $n-1$. Substituting $F_n\big(\frac{x}{y}\big)=\frac{1}{y^{n-1}}F_n(x,y)$ into the recurrence definition of the Fibonacci polynomials, we have
$$\frac{1}{y^{n-1}}F_n(x,y)=\frac{x}{y}\cdot\frac{1}{y^{n-2}}F_{n-1}(x,y)+\frac{1}{y^{n-3}}F_{n-2}(x,y).$$
Multiplying $y^{n-1}$ to both sides of the equation, we get $F_n(x,y)=xF_{n-1}(x,y)+y^2F_{n-2}(x,y)$, which we call the $n$-th \emph{homogenized Fibonacci polynomial}. Together with the observation that $F_1(x,y)=y^0F_1\big(\frac{x}{y}\big)=1$ and $F_2(x,y)=y^1F_2\big(\frac{x}{y}\big)=x$, we can easily see that a substitution of $x=s$ and $y^2=t$ yields the $n$-th bivariate Lucas polynomial.

Define $\Psi_1(x,y)=1$, and for all integers $n\geq2$, define $\Psi_n(x,y)=y^{\varphi(n)}\Psi_n\big(\frac{x}{y}\big)$ as the $n$-th \emph{homogenized fibotomic polynoimal}. It is easy to see that
$$F_n(x,y)=\prod_{d\mid n}\Psi_d(x,y)$$
for all positive integers $n$. The main results of this article are the following theorems regarding homogenized fibotomic polynomials.

\begin{theorem}\label{thm:homodiscriminant}
Let $n\geq2$ be an integer. Then the discriminant of $\Psi_n(x,y)$ with respect to $x$ is given by
$$\begin{cases}
\frac{(-1)^{\lfloor\varphi(n)/2\rfloor}(2n)^{\varphi(n)}y^{\varphi(n)(\varphi(n)-1)}}{p^{p^{\alpha-1}+1}}&\text{if $n=p^\alpha$ for some prime $p$ and some positive integer $\alpha$};\\
\frac{(-1)^{\varphi(n)/2}(2n)^{\varphi(n)}y^{\varphi(n)(\varphi(n)-1)}}{\underset{p\mid n}{\prod}p^{\varphi(n)/(p-1)}}&\text{otherwise}.
\end{cases}$$
\end{theorem}

Motivated by Lehmer \cite{lehmer} who determined the resultant of two cyclotomic polynomials, we obtain the following theorem.

\begin{theorem}\label{thm:homoresultant}
Let $2\leq m<n$ be integers. Then the resultant of $\Psi_m(x,y)$ and $\Psi_n(x,y)$ with respect to $x$ is given by
$$\begin{cases}
p^{\varphi(m)}y^{\varphi(m)\varphi(n)}&\text{if $n/m=p^\alpha$ for some prime $p$ and some positive integer $\alpha$};\\
y^{\varphi(m)\varphi(n)}&\text{otherwise}.
\end{cases}$$
\end{theorem}

Parallel to Guerrier's work \cite{guerrier} on completely determining the factorization form of $\Phi_n(x)$ in $\Z_p[x]$, we obtain the following theorem, which expands the study of Kitayama and Shiomi.

\begin{theorem}\label{thm:homofactorizationform}
Let $m$ be a positive integer such that $\gcd(p,m)=1$, and let $n=p^km$, where $k$ is a nonnegative integer.
\begin{itemize}
\item If $m=1$, then $\Psi_n(x,y)$ factors in $\Z_p[x]$ as $\left(x^2+4y^2\right)^{\frac{\varphi(p^k)}{2}}$.
\item If $m=2$, then $\Psi_n(x,y)$ factors in $\Z_p[x]$ as $x^{\varphi(p^k)}$.
\item If $m\geq3$ and $p>2$, then let $\delta$ be defined as in Theorem~$\ref{thm:fullchar}$. In this case, $\Psi_n(x,y)$ factors in $\Z_p[x,y]$ as a product of $\varphi(m)/\delta$ distinct irreducible monic polynomials of degree $\delta$, each raised to the $\varphi(p^k)$-th power.
\item If $m\geq3$ and $p=2$, then let $\delta$ be defined as in Theorem~$\ref{thm:fullchar2}$. In this case, $\Psi_n(x,y)$ factors in $\Z_p[x,y]$ as a product of $\varphi(m)/(2\delta)$ distinct irreducible monic polynomials of degree $\delta$, each raised to the $2\varphi(p^k)$-th power.
\end{itemize}
\end{theorem}

In this article, we will first focus on the fibotomic polynomials, and with a simple homogenization process, the results in Corollary~\ref{cor:discriminant}, Corollary~\ref{cor:resultant}, and Theorem~\ref{thm:factorizationform} can be generalized to Theorems~\ref{thm:homodiscriminant}, \ref{thm:homoresultant}, and \ref{thm:homofactorizationform}, respectively. For instance, note that all the roots of $\Psi_n(x,y)$ with respect to $x$ are the same as the roots of $\Psi_n(x)$ except that they have an extra factor $y$. Thus, the discriminant of $\Psi_n(x,y)$ with respect to $x$ is $y^{\varphi(n)(\varphi(n)-1)}\Delta(\Psi_n(x))$, and the resultant of $\Psi_m(x,y)$ and $\Psi_n(x,y)$ with respect to $x$ is $y^{\varphi(m)\varphi(n)}\res(\Psi_m(x),\Psi_n(x))$.

\section{Notation and preliminary results}\label{sec:prelims}

Let $\Phi_n(x)$ denote the $n$-th cyclotomic polynomial, and recall that the roots of $\Phi_n(x)$ are the primitive $n$-th roots of unity.  We denote an arbitrary primitive $k$-th root of unity by $\zeta_k$.  Then 
$$\Phi_n(x)=\prod_{\substack{1\leq s\leq n\\ \gcd(n,s)=1}}\left(x-e^{2\pi is/n}\right)=\prod_{\substack{1\leq s\leq n\\ \gcd(n,s)=1}}(x-\zeta_n^s).$$
Levy provided the root form for $\Psi_n(x)$ when $n\geq 2$:
$$\Psi_n(x)=\prod_{\substack{1\leq s\leq n\\ \gcd(s,n)=1}}\left(x-2i\cos\frac{\pi s}{n}\right)=\prod_{\substack{1\leq s\leq n\\ \gcd(s,n)=1}}\left(x-\zeta_4(\zeta_{2n}^s+\zeta_{2n}^{-s})\right),
$$
where the second equality follows from $2i\cos\frac{s\pi}{n}=i(e^{2\pi is/n}+e^{-2\pi is/n})$.

The following well-known theorems for cyclotomic polynomials will be useful in our study.

\begin{theorem}\label{thm:cyclotomicconst}
Let $n\geq2$ be an integer.  Then
$$\Phi_n(1)=\begin{cases}
p&\text{if $n=p^\alpha$ for some prime $p$ and some positive integer $\alpha$};\\
1&\text{otherwise}.
\end{cases}$$
\end{theorem}

\begin{theorem}\label{thm:cyclotomicidentityp=2}
Let $m\geq3$ be an odd integer.  Then 
$$\Phi_{2m}(x)=\Phi_m(-x).$$
\end{theorem}

\begin{theorem}\label{thm:cyclotomicidentity}
Let $p$ be a prime and $m$ be a positive integer. Then
\begin{equation}\label{eqn:Phi_pm,p|m}
\Phi_{pm}(x)=\Phi_m(x^p)\qquad\text{if }p\mid m
\end{equation}
and
$$\Phi_{pm}(x)=\frac{\Phi_m(x^p)}{\Phi_m(x)}\qquad\text{if }p\nmid m.$$
\end{theorem}

Let $\Delta(f(x))$ denote the discriminant of a polynomial $f$, and let $\res(f(x),g(x))$ denote the resultant of polynomials $f(x)$ and $g(x)$.

\begin{theorem}\label{thm:Phidiscrim}
Let $n$ be a positive integer. Then
$$\Delta(\Phi_n(x))=\frac{(-1)^{\lfloor\varphi(n)/2\rfloor}n^{\varphi(n)}}{\displaystyle\prod_{p\mid n}p^{\varphi(n)/(p-1)}}.$$
\end{theorem}

\begin{theorem}[\cite{lehmer}]\label{thm:Phires}
Let $m<n$ be positive integers. Then
$$\res(\Phi_m(x),\Phi_n(x))=\begin{cases}
p^{\varphi(m)}&\text{if $n/m=p^\alpha$ for some prime $p$ and some positive integer $\alpha$};\\
1&\text{otherwise}.
\end{cases}$$
\end{theorem}

\section{Identities involving $\Psi_n(x)$}

For the remainder of this paper, let $\omega=\frac{x+\sqrt{x^2+4}}{2}$. Levy gave a brief explanation for the following theorem. We provide a detailed proof here for completion.

\begin{theorem}\label{thm:Psi=Phi}
Let $n\geq2$ be an integer. Then
$$\Psi_n(x)=\begin{cases}
-\frac{\Phi_n(-\omega^2)}{\omega^{\varphi(n)}}&\text{if }n=2;\\
\frac{\Phi_n(-\omega^2)}{\omega^{\varphi(n)}}&\text{if }n\geq3.
\end{cases}$$
\end{theorem}

\begin{proof}
If $n=2$, then the statement holds since $\Psi_2(x)=x$ and
$$-\frac{\Phi_2(-\omega^2)}{\omega^{\varphi(2)}}=-\frac{-\omega^2+1}{\omega}=\omega-\omega^{-1}=\frac{x+\sqrt{x^2+4}}{2}-\frac{-x+\sqrt{x^2+4}}{2}=x.$$
Now, assume that $n\geq3$. It is well-known that $\sum_{d\mid n}\mu\left(\frac{n}{d}\right)=0$ for all integers $n\geq2$ and $\sum_{d\mid n}d\mu\left(\frac{n}{d}\right)=\varphi(n)$ for all positive integers $n$. Hence, for every nonzero $k$ that is independent of $d$,
\begin{equation}\label{eqn:mobius1}
\prod_{d\mid n}k^{\mu\left(\frac{n}{d}\right)}=1\text{ for all integers }n\geq2
\end{equation}
and
\begin{equation}\label{eqn:mobiusk^phin}
\prod_{d\mid n}\left(k^d\right)^{\mu\left(\frac{n}{d}\right)}=k^{\varphi(n)}\text{ for all positive integers }n.
\end{equation}

Since $x^n-1=\prod_{d\mid n}\Phi_d(x)$, M\"{o}bius inversion yields that for all real numbers $x\neq1$,
\begin{equation}\label{eqn:Phi}
\Phi_n(x)=\prod_{d\mid n}(x^d-1)^{\mu\left(\frac{n}{d}\right)}=\prod_{d\mid n}\left(\frac{x^d-1}{x-1}\right)^{\mu\left(\frac{n}{d}\right)}=\prod_{d\mid n}\left(\sum_{j=0}^{d-1}x^j\right)^{\mu\left(\frac{n}{d}\right)},
\end{equation}
where the second equality is due to \eqref{eqn:mobius1} by substituting $k=\frac{1}{x-1}$. By substituting $-\omega^2$ into \eqref{eqn:Phi}, we have
\begin{equation}\label{eqn:Phiuseful}
\Phi_n(-\omega^2)=\prod_{d\mid n}\left(\sum_{j=0}^{d-1}(-\omega^2)^j\right)^{\mu\left(\frac{n}{d}\right)}.
\end{equation}
Similarly, since $F_n(x)=\prod_{d\mid n}\Psi_d(x)$, M\"{o}bius inversion yields that for all real numbers $x\neq0$,
\begin{equation}\label{eqn:Psi}
\Psi_n(x)=\prod_{d\mid n}F_d(x)^{\mu\left(\frac{n}{d}\right)}.
\end{equation}
Note that $\varphi(n)$ is even for all integers $n\geq3$. Thus, for all real numbers $x\neq0$, \eqref{eqn:Psi} gives
\begin{equation}\label{eqn:Psiuseful}
\omega^{\varphi(n)}\Psi_n(x)=(-\omega)^{\varphi(n)}\Psi_n(x)=\prod_{d\mid n}((-\omega)^dF_d(x))^{\mu\left(\frac{n}{d}\right)}=\prod_{d\mid n}((-\omega)^{d-1}F_d(x))^{\mu\left(\frac{n}{d}\right)},
\end{equation}
where the second equality is due to \eqref{eqn:mobiusk^phin} and the third equality is due to \eqref{eqn:mobius1} by substituting $k=(-\omega)^{-1}$.

To prove our theorem, we first equate \eqref{eqn:Phiuseful} and \eqref{eqn:Psiuseful} for all real numbers $x\neq0$, and the proof will be complete by noticing that both $\Phi_n(-\omega^2)$ and $\omega^{\varphi(n)}\Psi_n(x)$ are continuous functions with respect to $x$. To equate \eqref{eqn:Phiuseful} and \eqref{eqn:Psiuseful}, it suffices to show that for all positive integers $d$,
\begin{equation}\label{eqn:sum=Fd}
\sum_{j=0}^{d-1}(-\omega^2)^j=(-\omega)^{d-1}F_d(x),
\end{equation}
and we shall proceed by induction.

When $d=1$, \eqref{eqn:sum=Fd} clearly holds since $F_1(x)=1$. When $d=2$, \eqref{eqn:sum=Fd} also holds since $$1-\omega^2=1-\frac{x^2+x^2+4+2x\sqrt{x^2+4}}{4}=-\frac{x^2+x\sqrt{x^2+4}}{2}=-\omega x=-\omega F_2(x).$$
Assuming that \eqref{eqn:sum=Fd} holds for some positive integers $d$ and $d+1$, we have
\begin{align*}
(-\omega)^{d+1}F_{d+2}(x)&=(-\omega)^{d+1}(xF_{d+1}(x)+F_d(x))\\
&=(-\omega x)(-\omega)^dF_{d+1}(x)+(-\omega)^2(-\omega)^{d-1}F_d(x)\\
&=(1-\omega^2)\left(\sum_{j=0}^d(-\omega^2)^j\right)+\omega^2\left(\sum_{j=0}^{d-1}(-\omega^2)^j\right)\\
&=\sum_{j=0}^d(-\omega^2)^j-\omega^2(-\omega^2)^d\\
&=\sum_{j=0}^{d+1}(-\omega^2)^j.
\end{align*}
Therefore, \eqref{eqn:sum=Fd} holds for all positive integers $d$ by induction.
\end{proof}

\begin{remark}\label{rmk:Psi=Phi}
As seen in the proof of 
Theorem~\ref{thm:Psi=Phi} that
\begin{equation}\label{eqn:omega-omegainv}
\omega-\omega^{-1}=x,
\end{equation}
the statement of Theorem~\ref{thm:Psi=Phi} can be rewritten such that for any integer $n\geq2$,
$$\Psi_n(\omega-\omega^{-1})=\begin{cases}
-\frac{\Phi_n(-\omega^2)}{\omega^{\varphi(n)}}&\text{if }n=2;\\
\frac{\Phi_n(-\omega^2)}{\omega^{\varphi(n)}}&\text{if }n\geq3.
\end{cases}$$
\end{remark}

The following two theorems are special cases of several results presented by Sagan and Tirrell.  We provide alternative proofs of these results using Theorem~\ref{thm:Psi=Phi}. Our version allows direct applications in Sections~\ref{sec:discriminant} and \ref{sec:factorize}.

\begin{theorem}
The constant term of the $n$-th fibotomic polynomial is given by
$$\Psi_n(0)=
\begin{cases}
0&\text{ if }n=2;\\
p&\text{ if $n=2p^{\alpha}$ for some prime $p$ and some positive integer $\alpha$};\\
1&\text{ otherwise}.
\end{cases}$$
\end{theorem}

\begin{proof}
Clearly, $\Psi_1(0)=1$ and $\Psi_2(0)=0$, so we now focus on $n\geq 3$.  If $x=0$, then $\omega=1$.  By Theorem~\ref{thm:Psi=Phi},
$$\Psi_n(0)=\frac{\Phi_n(-1)}{1^{\varphi(n)}}=\Phi_{n}(-1).$$
Suppose that $n=2m$ is even with $m\geq 2$.  If $m$ is odd, then $\Phi_{2m}(-1)=\Phi_m(1)$ by Theorem~\ref{thm:cyclotomicidentityp=2}, and if $m$ is even, then $\Phi_{2m}(-1)=\Phi_m((-1)^2)=\Phi_m(1)$ by Theorem~\ref{thm:cyclotomicidentity}.  In both cases, the result follows from Theorem~\ref{thm:cyclotomicconst}.

Now, suppose that $n$ is odd and let $n=p_1^{\alpha_1}p_2^{\alpha_2}\cdots p_r^{\alpha_r}$ be the prime factorization of $n$.  By repeated use of Theorem~\ref{thm:cyclotomicidentity},
$$\Phi_n(-1)=\Phi_{p_1p_2\cdots p_r}\left((-1)^{p_1^{\alpha_1-1}p_2^{\alpha_2-1}\cdots p_r^{\alpha_r-1}}\right)=\Phi_{p_1p_2\cdots p_r}(-1)=\frac{\Phi_{p_1p_2\cdots p_{r-1}}\left((-1)^{p_r}\right)}{\Phi_{p_1p_2\cdots p_{r-1}}(-1)}=1.$$
\end{proof}

The next theorem provides a number of identities that are parallel to Theorems~\ref{thm:cyclotomicidentityp=2} and \ref{thm:cyclotomicidentity}.

\begin{theorem}\label{thm:Psi_pm}
Let $p$ be a prime and $m\geq2$ be an integer.
\begin{enumerate}[\indent$(a)$]
\item\label{thm:Psi_2m,2|m} If $p=2$ and $p\mid m$, then 
$$\Psi_{2m}(x)=\begin{cases}
-i^{\varphi(m)}\Psi_m\left(i\omega^2-(i\omega^2)^{-1}\right)=-i^{\varphi(m)}\Psi_m(i(x^2+2))&\text{if }m=2;\\
i^{\varphi(m)}\Psi_m\left(i\omega^2-(i\omega^2)^{-1}\right)=i^{\varphi(m)}\Psi_m(i(x^2+2))&\text{if }m\geq3.
\end{cases}$$
\item\label{thm:Psi_2m,2not|m} If $p=2$ and $p\nmid m$, then
$$\Psi_{2m}(x)=i^{\varphi(m)}\Psi_m\left(i\omega-(i\omega)^{-1}\right)=i^{\varphi(m)}\Psi_m\left(i\sqrt{x^2+4}\right).$$
\item\label{thm:Psi_pm,p|m} If $p>2$ and $p\mid m$, then
$$\Psi_{pm}(x)=\Psi_m\left(\omega^p-\omega^{-p}\right)=\Psi_m(x\Psi_{2p}(x)).$$
\item\label{thm:Psi_pm,pnot|m} If $p>2$ and $p\nmid m$, then
$$\Psi_{pm}(x)=\frac{\Psi_m\left(\omega^p-\omega^{-p}\right)}{\Psi_m\left(\omega-\omega^{-1}\right)}=\frac{\Psi_m(x\Psi_{2p}(x))}{\Psi_m(x)}.$$
\end{enumerate}
\end{theorem}

\begin{proof}
$(\ref{thm:Psi_2m,2|m})$ If $2\mid m$, then
$$\Psi_{2m}(x)=\frac{\Phi_{2m}(-\omega^2)}{\omega^{\varphi(2m)}}=\frac{\Phi_m(\omega^4)}{\omega^{2\varphi(m)}}=i^{\varphi(m)}\frac{\Phi_m(-(i\omega^2)^2)}{(i\omega^2)^{\varphi(m)}},$$
where the first equality is due to Theorem~\ref{thm:Psi=Phi}, and the second equality is due to equation~\eqref{eqn:Phi_pm,p|m}. By Remark~\ref{rmk:Psi=Phi}, if $m=2$, then
$$\Psi_{2m}(x)=-i^{\varphi(m)}\Psi_m\left(i\omega^2-(i\omega^2)^{-1}\right)=-i^{\varphi(m)}\Psi_m\left(i\left(\omega^2+\omega^{-2}\right)\right)=-i^{\varphi(m)}\Psi_m\left(i(x^2+2)\right),$$
and if $m\geq3$, then
$$\Psi_{2m}(x)=i^{\varphi(m)}\Psi_m\left(i\omega^2-(i\omega^2)^{-1}\right)=i^{\varphi(m)}\Psi_m\left(i\left(\omega^2+\omega^{-2}\right)\right)=i^{\varphi(m)}\Psi_m\left(i(x^2+2)\right).$$

$(\ref{thm:Psi_2m,2not|m})$ If $2\nmid m$, then $m\geq3$, and
$$\Psi_{2m}(x)=\frac{\Phi_{2m}(-\omega^2)}{\omega^{\varphi(2m)}}=\frac{\Phi_m(\omega^2)}{\omega^{\varphi(m)}}=i^{\varphi(m)}\frac{\Phi_m(-(i\omega)^2)}{(i\omega)^{\varphi(m)}},$$
where the second equality is due to Theorem~\ref{thm:cyclotomicidentityp=2}. Again by Remark~\ref{rmk:Psi=Phi}, we have
$$\Psi_{2m}(x)=i^{\varphi(m)}\Psi_m\left(i\omega-(i\omega)^{-1}\right)=i^{\varphi(m)}\Psi_m\left(i\left(\omega+\omega^{-1}\right)\right)=i^{\varphi(m)}\Psi_m\left(i\sqrt{x^2+4}\right),$$

$(\ref{thm:Psi_pm,p|m})$ If $p>2$ and $p\mid m$, then $m\geq3$, and
\begin{equation}\label{eqn:Psi_pmhalffinish}
\Psi_{pm}(x)=\frac{\Phi_{pm}(-\omega^2)}{\omega^{\varphi(pm)}}=\frac{\Phi_m((-\omega^2)^p)}{\omega^{p\varphi(m)}}=\frac{\Phi_m(-(\omega^p)^2)}{(\omega^p)^{\varphi(m)}}=\Psi_m\left(\omega^p-\omega^{-p}\right).
\end{equation}
From Webb and Parberry \cite{wp},
\begin{equation*}\label{eqn:F_n omega}
F_n(x)=\frac{\omega^n-(-\omega)^{-n}}{\omega+\omega^{-1}}
\end{equation*}
for all positive integers $n$, so
$$\omega^p-\omega^{-p}=\frac{F_{2p}(x)}{F_p(x)}=\frac{\Psi_2(x)\Psi_p(x)\Psi_{2p}(x)}{\Psi_p(x)}=x\Psi_{2p}(x).$$
Substituting this result into equation~\eqref{eqn:Psi_pmhalffinish} yields
$$\Psi_{pm}(x)=\Psi_m(x\Psi_{2p}(x)).$$

$(\ref{thm:Psi_pm,pnot|m})$ If $p>2$ and $p\nmid m$, then
$$\Psi_{pm}(x)=\frac{\Phi_{pm}(-\omega^2)}{\omega^{\varphi(pm)}}=\frac{\Phi_m((-\omega^2)^p)}{\Phi_m(-\omega^2)\omega^{(p-1)\varphi(m)}}=\left.\frac{\Phi_m(-(\omega^p)^2)}{(\omega^p)^{\varphi(m)}}\right/\frac{\Phi_m(-\omega^2)}{\omega^{\varphi(m)}},$$
which is equal to 
$$\frac{\Psi_m(\omega^p-\omega^{-p})}{\Psi_m(\omega-\omega^{-1})}=\frac{\Psi_m(x\Psi_{2p}(x))}{\Psi_m(x)}$$
for both $m=2$ and $m\geq3$.
\end{proof}

\section{Discriminant and resultant formulas}\label{sec:discriminant}

Our goal in this section is to provide the formulas of the discriminant $\Delta(\Psi_n(x))$ for all integers $n\geq2$ and the resultant $\res(\Psi_m(x),\Psi_n(x))$ for all integers $2\leq m<n$. To achieve this, we compare $\Delta(\Psi_n(x))$ and $\res(\Psi_m(x),\Psi_n(x))$ with their cyclotomic counterparts. We begin with studying the discriminants.
  
\begin{theorem}\label{thm:Psi/Phidiscrim}
Let $n\geq2$ be an integer. Then
$$\frac{\Delta(\Psi_n(x))}{\Delta(\Phi_n(x))}=\begin{cases}
\frac{2^{\varphi(n)}}{p}&\text{if $n=p^\alpha$ for some prime $p$ and some positive integer $\alpha$};\\
2^{\varphi(n)}&\text{otherwise}.
\end{cases}$$
\end{theorem}

\begin{proof}
Here are three elementary trigonometric identities that we will use in this proof:
\begin{equation}\label{eqn:trigcos-cos}
\cos x-\cos y=-2\sin\frac{x+y}{2}\sin\frac{x-y}{2},
\end{equation}
\begin{equation}\label{eqn:trigdouble}
\sin2x=2\sin x\cos x,
\end{equation}
and
\begin{equation}\label{eqn:trigcomplementary}
\cos x=\sin\left(\frac{\pi}{2}-x\right).
\end{equation}
Another useful identity is
\begin{equation}\label{eqn:sums}
\sum_{\substack{1\leq s\leq n-1\\\gcd(s,n)=1}}s=\frac{n\varphi(n)}{2}
\end{equation}
for all integers $n\geq2$. Using this identity, we find that for all integers $n\geq2$,
\begin{equation}\label{eqn:s+t}
\begin{split}
\sum_{\substack{1\leq s\neq t\leq n-1\\\gcd(s,n)=\gcd(t,n)=1}}\hspace{-15pt}(s+t)&=\sum_{\substack{1\leq s,t\leq n-1\\\gcd(s,n)=\gcd(t,n)=1}}\hspace{-15pt}(s+t)-\sum_{\substack{1\leq s=t\leq n-1\\\gcd(s,n)=\gcd(t,n)=1}}\hspace{-15pt}(s+t)\\
&=\sum_{\substack{1\leq s\leq n-1\\\gcd(s,n)=1}}\Bigg(\sum_{\substack{1\leq t\leq n-1\\\gcd(t,n)=1}}s+\sum_{\substack{1\leq t\leq n-1\\\gcd(t,n)=1}}t\Bigg)-\sum_{\substack{1\leq s\leq n-1\\\gcd(s,n)=1}}2s\\
&=\sum_{\substack{1\leq s\leq n-1\\\gcd(s,n)=1}}\left(s\varphi(n)+\frac{n\varphi(n)}{2}\right)-2\cdot\frac{n\varphi(n)}{2}\\
&=\frac{n\varphi(n)}{2}\cdot\varphi(n)+\frac{n\varphi(n)}{2}\cdot\varphi(n)-n\varphi(n)\\
&=n\varphi(n)(\varphi(n)-1).
\end{split}
\end{equation}

With these identities established, we now begin the proof of this theorem. By the definition of the discriminant of a polynomial, we obtain
\begin{align*}
\Delta(\Psi_n(x))&=\prod_{\substack{1\leq s\neq t\leq n-1\\\gcd(s,n)=\gcd(t,n)=1}}\hspace{-15pt}\left(2i\cos\frac{s\pi}{n}-2i\cos\frac{t\pi}{n}\right)\\
&=\prod_{\substack{1\leq s\neq t\leq n-1\\\gcd(s,n)=\gcd(t,n)=1}}\hspace{-15pt}\left(2i\left(-2\sin\frac{(s+t)\pi}{2n}\sin\frac{(s-t)\pi}{2n}\right)\right)\qquad\text{(by \eqref{eqn:trigcos-cos})}\\
&=(-1)^{\varphi(n)^2-\varphi(n)}\cdot\prod_{\substack{1\leq s\neq t\leq n-1\\\gcd(s,n)=\gcd(t,n)=1}}\hspace{-15pt}\left(2i\cdot2\sin\frac{(s+t)\pi}{2n}\sin\frac{(s-t)\pi}{2n}\right)\\
&=\prod_{\substack{1\leq s\neq t\leq n-1\\\gcd(s,n)=\gcd(t,n)=1}}\hspace{-15pt}\left(2i\cdot2\sin\frac{(s+t)\pi}{2n}\sin\frac{(s-t)\pi}{2n}\right)
\end{align*}
and
\begin{align*}
\Delta(\Phi_n(x))&=\prod_{\substack{1\leq s\neq t\leq n-1\\\gcd(s,n)=\gcd(t,n)=1}}\hspace{-15pt}\left(e^{\frac{2is\pi}{n}}-e^{\frac{2it\pi}{n}}\right)\\
&=\prod_{\substack{1\leq s\neq t\leq n-1\\\gcd(s,n)=\gcd(t,n)=1}}\hspace{-15pt}e^{\frac{2i(s+t)\pi}{2n}}\left(e^{\frac{2i(s-t)\pi}{2n}}-e^{-\frac{2i(s-t)\pi}{2n}}\right)\\
&=\prod_{\substack{1\leq s\neq t\leq n-1\\\gcd(s,n)=\gcd(t,n)=1}}\hspace{-15pt}e^{\frac{2i(s+t)\pi}{2n}}\cdot\prod_{\substack{1\leq s\neq t\leq n-1\\\gcd(s,n)=\gcd(t,n)=1}}\hspace{-15pt}\left(e^{\frac{2i(s-t)\pi}{2n}}-e^{-\frac{2i(s-t)\pi}{2n}}\right)\\
&=\exp\Bigg(\frac{i\pi}{n}\sum_{\substack{1\leq s\neq t\leq n-1\\\gcd(s,n)=\gcd(t,n)=1}}\hspace{-15pt}(s+t)\Bigg)\cdot\prod_{\substack{1\leq s\neq t\leq n-1\\\gcd(s,n)=\gcd(t,n)=1}}\hspace{-15pt}\left(2i\sin\frac{2(s-t)\pi}{2n}\right)\\
&=\exp\left(\frac{i\pi}{n}\cdot n\varphi(n)(\varphi(n)-1)\right)\cdot\prod_{\substack{1\leq s\neq t\leq n-1\\\gcd(s,n)=\gcd(t,n)=1}}\hspace{-15pt}\left(2i\sin\frac{2(s-t)\pi}{2n}\right)\qquad\text{(by \eqref{eqn:s+t})}\\
&=\prod_{\substack{1\leq s\neq t\leq n-1\\\gcd(s,n)=\gcd(t,n)=1}}\hspace{-15pt}\left(2i\cdot2\sin\frac{(s-t)\pi}{2n}\cos\frac{(s-t)\pi}{2n}\right).\qquad\text{(by \eqref{eqn:trigdouble})}
\end{align*}

Next, we compute the ratio of the discriminants to obtain
\begin{align*}
\frac{\Delta(\Psi_n(x))}{\Delta(\Phi_n(x))}&=\frac{\displaystyle\prod_{\substack{1\leq s\neq t\leq n-1\\\gcd(s,n)=\gcd(t,n)=1}}\hspace{-15pt}\sin\frac{(s+t)\pi}{2n}}{\displaystyle\prod_{\substack{1\leq s\neq t\leq n-1\\\gcd(s,n)=\gcd(t,n)=1}}\hspace{-15pt}\cos\frac{(s-t)\pi}{2n}}\\
&=\frac{\displaystyle\prod_{\substack{1\leq s\neq t\leq n-1\\\gcd(s,n)=\gcd(t,n)=1}}\hspace{-15pt}\sin\frac{(s+t)\pi}{2n}}{\displaystyle\prod_{\substack{1\leq s\neq t\leq n-1\\\gcd(s,n)=\gcd(t,n)=1}}\hspace{-15pt}\sin\frac{((n-s)+t)\pi}{2n}}\qquad\text{(by \eqref{eqn:trigcomplementary})}\\
&=\frac{\displaystyle\left.\prod_{\substack{1\leq s,t\leq n-1\\\gcd(s,n)=\gcd(t,n)=1}}\hspace{-15pt}\sin\frac{(s+t)\pi}{2n}\right/\prod_{\substack{1\leq s=t\leq n-1\\\gcd(s,n)=\gcd(t,n)=1}}\hspace{-15pt}\sin\frac{(s+t)\pi}{2n}}{\displaystyle\left.\prod_{\substack{1\leq s,t\leq n-1\\\gcd(s,n)=\gcd(t,n)=1}}\hspace{-15pt}\sin\frac{((n-s)+t)\pi}{2n}\right/\prod_{\substack{1\leq s=t\leq n-1\\\gcd(s,n)=\gcd(t,n)=1}}\hspace{-15pt}\sin\frac{((n-s)+t)\pi}{2n}}.
\end{align*}
Note that the numerator of the numerator and the numerator of the denominator are the same product, and the denominator of the denominator evaluates to $1$, so we have
\begin{equation}\label{eqn:prodsin}
\frac{\Delta(\Psi_n(x))}{\Delta(\Phi_n(x))}=\frac{1}{\displaystyle\prod_{\substack{1\leq s=t\leq n-1\\\gcd(s,n)=\gcd(t,n)=1}}\hspace{-15pt}\sin\frac{(s+t)\pi}{2n}}=\frac{1}{\displaystyle\prod_{\substack{1\leq s\leq n-1\\\gcd(s,n)=1}}\sin\frac{s\pi}{n}}.
\end{equation}

Let $n=p_1^{\alpha_1}p_2^{\alpha_2}\dotsb p_r^{\alpha_r}$ be the unique prime factorization of $n$. By using the trigonometric identity 
$$\prod_{1\leq s\leq n-1}\sin\frac{s\pi}{n}=\frac{n}{2^{n-1}}=\frac{2n}{2^n}$$
offered in MathWorld \cite{sine}, we see for each $1\leq\ell\leq r$ that
$$\prod_{\substack{1\leq s\leq n-1\\p_{j_1}p_{j_2}\dotsb p_{j_\ell}\mid s}}\sin\frac{s\pi}{n}=\prod_{1\leq s\leq \frac{n}{p_{j_1}p_{j_2}\dotsb p_{j_\ell}}-1}\sin\frac{s\pi}{\;\frac{n}{p_{j_1}p_{j_2}\dotsb p_{j_\ell}}\;}=\frac{\frac{2n}{p_{j_1}p_{j_2}\dotsb p_{j_\ell}}}{\;2^{\frac{n}{p_{j_1}p_{j_2}\dotsb p_{j_\ell}}}\;}.$$
By the inclusion-exclusion principle, we have
$$\prod_{\substack{1\leq s\leq n-1\\\gcd(s,n)=1}}\sin\frac{s\pi}{n}=\frac{\displaystyle\frac{2n}{2^n}\cdot\prod_{1\leq j_1<j_2\leq r}\frac{\frac{2n}{p_{j_1}p_{j_2}}}{\;2^{\frac{n}{p_{j_1}p_{j_2}}}\;}\cdot\prod_{1\leq j_1<j_2<j_3<j_4\leq r}\frac{\frac{2n}{p_{j_1}p_{j_2}p_{j_3}p_{j_4}}}{\;2^{\frac{n}{p_{j_1}p_{j_2}p_{j_3}p_{j_4}}}\;}\dotsb}{\displaystyle\prod_{1\leq j_1\leq r}\frac{\frac{2n}{p_{j_1}}}{\;2^{\frac{n}{p_{j_1}}}\;}\cdot\prod_{1\leq j_1<j_2<j_3\leq r}\frac{\frac{2n}{p_{j_1}p_{j_2}p_{j_3}}}{\;2^{\frac{n}{p_{j_1}p_{j_2}p_{j_3}}}\;}\dotsb}.$$
The number of occurrences of the factor $2n$ in the numerator is equal to the sum of the positive coefficients in the binomial expansion of $(x-1)^r$, while the number of occurrences of the factor $2n$ in the denominator is equal to the sum of the negative coefficients in the binomial expansion of $(x-1)^r$. Hence, they cancel out each other completely. The number of occurrences of the factor $\frac{1}{p_j}$ in the numerator is equal to the sum of the negative coefficients in the binomial expansion of $(x-1)^{r-1}$, while the number of occurrences of the factor $\frac{1}{p_j}$ in the denominator is equal to the sum of the positive coefficients in the binomial expansion of $(x-1)^{r-1}$. Hence, they cancel out each other completely when $r\geq2$, but the factor $\frac{1}{p_1}$ remains in the denominator when $r=1$. Lastly, it is easy to see that the exponent of the factor $\frac{1}{2}$ is precisely $\varphi(n)$. Therefore,
$$\prod_{\substack{1\leq s\leq n-1\\\gcd(s,n)=1}}\sin\frac{s\pi}{n}=\begin{cases}
\frac{\;\frac{1}{2^{\varphi(n)}}\;}{\frac{1}{p_1}}&\text{if }r=1;\\
\frac{1}{2^{\varphi(n)}}&\text{otherwise},
\end{cases}$$
which completes our proof by substituting this last equation into \eqref{eqn:prodsin}.
\end{proof}

The following is a corollary of Theorems~\ref{thm:Phidiscrim} and \ref{thm:Psi/Phidiscrim}, which expresses $\Delta(\Psi_n(x))$ as a closed form in a similar manner to $\Delta(\Phi_n(x))$.

\begin{corollary}\label{cor:discriminant}
Let $n\geq2$ be an integer. Then
$$\Delta(\Psi_n(x))=\begin{cases}
\frac{(-1)^{\lfloor\varphi(n)/2\rfloor}(2n)^{\varphi(n)}}{p^{p^{\alpha-1}+1}}&\text{if $n=p^\alpha$ for some prime $p$ and some positive integer $\alpha$};\\
\frac{(-1)^{\varphi(n)/2}(2n)^{\varphi(n)}}{\underset{p\mid n}{\prod}p^{\varphi(n)/(p-1)}}&\text{otherwise}.
\end{cases}$$
\end{corollary}

We finish this section by studying the resultants.

\begin{theorem}\label{thm:Psi/Phires}
Let $2\leq m<n$ be integers. Then
$$\frac{\res(\Psi_m(x),\Psi_n(x))}{\res(\Phi_m(x),\Phi_n(x))}=1.$$
\end{theorem}

\begin{proof}
Using identity \eqref{eqn:sums}, we find that for all integers $m,n\geq2$,
\begin{equation}\label{eqn:s/m+t/n}
\begin{split}
\sum_{\substack{1\leq s\leq m-1\\1\leq t\leq n-1\\\gcd(s,m)=\gcd(t,n)=1}}\hspace{-15pt}\left(\frac{s}{m}+\frac{t}{n}\right)&=\sum_{\substack{1\leq s\leq m-1\\\gcd(s,m)=1}}\Bigg(\sum_{\substack{1\leq t\leq n-1\\\gcd(t,n)=1}}\frac{s}{m}+\sum_{\substack{1\leq t\leq n-1\\\gcd(t,n)=1}}\frac{t}{n}\Bigg)\\
&=\sum_{\substack{1\leq s\leq m-1\\\gcd(s,m)=1}}\left(\frac{s}{m}\cdot\varphi(n)+\frac{1}{n}\cdot\frac{n\varphi(n)}{2}\right)\\
&=\frac{1}{m}\cdot\frac{m\varphi(m)}{2}\cdot\varphi(n)+\varphi(m)\cdot\frac{\varphi(n)}{2}\\
&=\varphi(m)\varphi(n).
\end{split}
\end{equation}

With this identity established, we now begin the proof of this theorem. By the definition of the resultant of two polynomials, we obtain
\begin{align*}
&\hspace{18pt}\res(\Psi_m(x),\Psi_n(x))\\
&=\prod_{\substack{1\leq s\leq m-1\\1\leq t\leq n-1\\\gcd(s,m)=\gcd(t,n)=1}}\hspace{-15pt}\left(2i\cos\frac{s\pi}{m}-2i\cos\frac{t\pi}{n}\right)\\
&=\prod_{\substack{1\leq s\leq m-1\\1\leq t\leq n-1\\\gcd(s,m)=\gcd(t,n)=1}}\hspace{-15pt}\left(2i\left(-2\sin\left(\frac{s\pi}{2m}+\frac{t\pi}{2n}\right)\sin\left(\frac{s\pi}{2m}-\frac{t\pi}{2n}\right)\right)\right)\qquad\text{(by \eqref{eqn:trigcos-cos})}\\
&=(-1)^{\varphi(m)\varphi(n)}\cdot\prod_{\substack{1\leq s\leq m-1\\1\leq t\leq n-1\\\gcd(s,m)=\gcd(t,n)=1}}\hspace{-15pt}\left(2i\cdot2\sin\left(\frac{s\pi}{2m}+\frac{t\pi}{2n}\right)\sin\left(\frac{s\pi}{2m}-\frac{t\pi}{2n}\right)\right)\\
&=\prod_{\substack{1\leq s\leq m-1\\1\leq t\leq n-1\\\gcd(s,m)=\gcd(t,n)=1}}\hspace{-15pt}\left(2i\cdot2\sin\left(\frac{s\pi}{2m}+\frac{t\pi}{2n}\right)\sin\left(\frac{s\pi}{2m}-\frac{t\pi}{2n}\right)\right)
\end{align*}
and
\begin{align*}
&\hspace{18pt}\res(\Phi_m(x),\Phi_n(x))\\
&=\prod_{\substack{1\leq s\leq m-1\\1\leq t\leq n-1\\\gcd(s,m)=\gcd(t,n)=1}}\hspace{-15pt}\left(e^{\frac{2is\pi}{m}}-e^{\frac{2it\pi}{n}}\right)\\
&=\prod_{\substack{1\leq s\leq m-1\\1\leq t\leq n-1\\\gcd(s,m)=\gcd(t,n)=1}}\hspace{-15pt}e^{2i\left(\frac{s\pi}{2m}+\frac{t\pi}{2n}\right)}\left(e^{2i\left(\frac{s\pi}{2m}-\frac{t\pi}{2n}\right)}-e^{-2i\left(\frac{s\pi}{2m}-\frac{t\pi}{2n}\right)}\right)\\
&=\prod_{\substack{1\leq s\leq m-1\\1\leq t\leq n-1\\\gcd(s,m)=\gcd(t,n)=1}}\hspace{-15pt}e^{2i\left(\frac{s\pi}{2m}+\frac{t\pi}{2n}\right)}\cdot\prod_{\substack{1\leq s\leq m-1\\1\leq t\leq n-1\\\gcd(s,m)=\gcd(t,n)=1}}\hspace{-15pt}\left(e^{2i\left(\frac{s\pi}{2m}-\frac{t\pi}{2n}\right)}-e^{-2i\left(\frac{s\pi}{2m}-\frac{t\pi}{2n}\right)}\right)\\
&=\exp\Bigg(i\pi\cdot\sum_{\substack{1\leq s\leq m-1\\1\leq t\leq n-1\\\gcd(s,m)=\gcd(t,n)=1}}\hspace{-15pt}\left(\frac{s}{m}+\frac{t}{n}\right)\Bigg)\cdot\prod_{\substack{1\leq s\leq m-1\\1\leq t\leq n-1\\\gcd(s,m)=\gcd(t,n)=1}}\hspace{-15pt}\left(2i\sin\left(2\left(\frac{s\pi}{2m}-\frac{t\pi}{2n}\right)\right)\right)\\
&=\exp\left(i\pi\cdot\varphi(m)\varphi(n)\right)\cdot\prod_{\substack{1\leq s\leq m-1\\1\leq t\leq n-1\\\gcd(s,m)=\gcd(t,n)=1}}\hspace{-15pt}\left(2i\sin\left(2\left(\frac{s\pi}{2m}-\frac{t\pi}{2n}\right)\right)\right)\qquad\text{(by \eqref{eqn:s/m+t/n})}\\
&=\prod_{\substack{1\leq s\leq m-1\\1\leq t\leq n-1\\\gcd(s,m)=\gcd(t,n)=1}}\hspace{-15pt}\left(2i\cdot2\sin\left(\frac{s\pi}{2m}-\frac{t\pi}{2n}\right)\cos\left(\frac{s\pi}{2m}-\frac{t\pi}{2n}\right)\right).\qquad\text{(by \eqref{eqn:trigdouble})}
\end{align*}

Next, we compute the ratio of the resultants to obtain
\begin{align*}
\frac{\res(\Psi_m(x),\Psi_n(x))}{\res(\Phi_m(x),\Phi_n(x))}&=\frac{\displaystyle\prod_{\substack{1\leq s\leq m-1\\1\leq t\leq n-1\\\gcd(s,m)=\gcd(t,n)=1}}\hspace{-15pt}\sin\left(\frac{s\pi}{2m}+\frac{t\pi}{2n}\right)}{\displaystyle\prod_{\substack{1\leq s\leq m-1\\1\leq t\leq n-1\\\gcd(s,m)=\gcd(t,n)=1}}\hspace{-15pt}\cos\left(\frac{s\pi}{2m}-\frac{t\pi}{2n}\right)}\\
&=\frac{\displaystyle\prod_{\substack{1\leq s\leq m-1\\1\leq t\leq n-1\\\gcd(s,m)=\gcd(t,n)=1}}\hspace{-15pt}\sin\left(\frac{s\pi}{2m}+\frac{t\pi}{2n}\right)}{\displaystyle\prod_{\substack{1\leq s\leq m-1\\1\leq t\leq n-1\\\gcd(s,m)=\gcd(t,n)=1}}\hspace{-15pt}\sin\left(\frac{(m-s)\pi}{2m}+\frac{t\pi}{2n}\right)}\qquad\text{(by \eqref{eqn:trigcomplementary})}\\
&=1.
\end{align*}
\end{proof}

The following is a corollary of Theorems~\ref{thm:Phires} and \ref{thm:Psi/Phires}, which expresses $\res(\Psi_m(x),\Psi_n(x))$ as a closed form in a similar manner to $\res(\Phi_m(x),\Phi_n(x))$.

\begin{corollary}\label{cor:resultant}
Let $2\leq m<n$ be integers. Then
$$\res(\Psi_m(x),\Psi_n(x))=\begin{cases}
p^{\varphi(m)}&\text{if $n/m=p^\alpha$ for some prime $p$ and some positive integer $\alpha$};\\
1&\text{otherwise}.
\end{cases}$$
\end{corollary}

\section{Factorization of $\Psi_n(x)$ in $\mathbb{Z}_p[x]$}\label{sec:factorize}

Throughout this section, let $p$ be a prime and let $\Z_p$ denote the prime field of size $p$. Furthermore, we view $\zeta_n$ as a solution to the congruence $\Phi_n(x)\mod{0}{p}$. Although $\zeta_n$ depends on $p$, we suppress the index $p$ to simplify the notation.  Our goal in this section is to provide the factorization form of $\Psi_n(x)$ in $\mathbb{Z}_p$.  We begin by considering the case $n=p^k$ for some positive integer $k$.

\begin{theorem}\label{thm:Psi_p^k}
Let $k$ be a positive integer. Then
$$\Psi_{p^k}(x)\mod{\big(x^2+4\big)^{\frac{\varphi(p^k)}{2}}}{p}.$$
\end{theorem}

\begin{proof}
If $p=2$ and $k=1$, then the identity holds trivially. Otherwise, notice that
$$\Psi_{p^k}(x)=\frac{\Phi_{p^k}(-\omega^2)}{\omega^{\varphi(p^k)}}=\frac{\Phi_p\Big(\big(-\omega^2\big)^{p^{k-1}}\Big)}{\omega^{\varphi(p^k)}}=\frac{\big(-\omega^2\big)^{p^k}-1}{\big(-\omega^2\big)^{p^{k-1}}-1}\cdot\frac{1}{\omega^{p^k-p^{k-1}}}=\frac{(-\omega)^{p^k}-\omega^{-p^k}}{(-\omega)^{p^{k-1}}-\omega^{-p^{k-1}}},$$
where the first equality is due to Theorem~\ref{thm:Psi=Phi}, the second equality is due to repeated application of equation~\eqref{eqn:Phi_pm,p|m}, and the third equality follows from $\Phi_p(x)=\frac{x^p-1}{x-1}$. As a result,
$$\Psi_{p^k}(x)=\frac{(-\omega)^{p^k}-\omega^{-p^k}}{(-\omega)^{p^{k-1}}-\omega^{-p^{k-1}}}\equiv\frac{(-\omega-\omega^{-1})^{p^k}}{(-\omega-\omega^{-1})^{p^{k-1}}}\mod{(-\omega-\omega^{-1})^{\varphi(p^k)}}{p},$$
and the theorem follows by noticing that $\varphi(p^k)$ is even and $\omega+\omega^{-1}=\sqrt{x^2+4}$.
\end{proof}

We next relate the factorization of $\Psi_n(x)$ to $\Psi_m(x)$, where $n=p^km$ and $\gcd(p,m)=1$.

\begin{theorem}\label{thm:Psi_p^km}
Let $m\geq2$ be an integer such that $\gcd(p,m)=1$, and let $k$ be a nonnegative integer. Then
$$\Psi_{p^km}(x)\mod{\Psi_m(x)^{\varphi(p^k)}}{p}.$$
\end{theorem}

\begin{proof}
This theorem holds trivially if $k=0$, so we assume for the rest of the proof that $k>0$. If $p=2$, then note that the only fourth root of unity over $\Z_2$ is $1$. Hence, $$\Psi_{2^km}(x)\equiv\Psi_{2m}\left(x^{2^{k-1}}\right)\equiv\Psi_m\left(x^{2^{k-1}}\right)\equiv\Psi_m(x)^{2^{k-1}}\mod{\Psi_m(x)^{\varphi(2^k)}}{2},$$
where the first equality is due to repeated application of Theorem~\ref{thm:Psi_pm}$(\ref{thm:Psi_2m,2|m})$, and the second equality is due to Theorem~\ref{thm:Psi_pm}$(\ref{thm:Psi_2m,2not|m})$.

If $p>2$, then 
\begin{align*}
\Psi_{p^km}(x)
&=\Psi_{pm}\left(\omega^{p^{k-1}}-\omega^{-p^{k-1}}\right)\\
&=\frac{\Psi_m\left(\omega^{p^k}-\omega^{-p^k}\right)}{\Psi_m\left(\omega^{p^{k-1}}-\omega^{-p^{k-1}}\right)}\\
&\equiv\frac{\Psi_m\left(\omega-\omega^{-1}\right)^{p^k}}{\Psi_m\left(\omega-\omega^{-1}\right)^{p^{k-1}}}\\
&\equiv\frac{\Psi_m(x)^{p^k}}{\Psi_m(x)^{p^{k-1}}}\mod{\Psi_m(x)^{\varphi(p^k)}}{p},
\end{align*}
where the first equality is due to repeated application of Theorem~\ref{thm:Psi_pm}$(\ref{thm:Psi_pm,p|m})$, the second equality is due to Theorem~\ref{thm:Psi_pm}$(\ref{thm:Psi_pm,pnot|m})$, and the second congruence is by \eqref{eqn:omega-omegainv}.
\end{proof}

Theorem~\ref{thm:Psi_p^km} allows us to focus our attention on the factorization of $\Psi_m(x)$ in $\Z_p[x]$ when $\gcd(p,m)=1$.  To move in this direction, we first present the following two lemmas.

\begin{lemma}\label{lem:galois}
Let $\alpha$ be algebraic (and thus separable) over $\Z_p$. Then the degree of the minimal polynomial of $\alpha$ over $\Z_p$ is the smallest positive integer $\delta$ such that $\alpha^{p^\delta}=\alpha$.
\end{lemma}

\begin{proof}
Since the Frobenius map $x\mapsto x^p$ is a generator of the Galois group of $\Z_p(\alpha)/\Z_p$ and $\delta$ is the smallest positive integer such that $\alpha^{p^\delta}=\alpha$, we know that $\left\{\alpha^{p^j}:0\leq j\leq\delta-1\right\}$ is the smallest subset of $\Z_p(\alpha)$ that contains $\alpha$ and is fixed by the Galois group. As a result,
$$\prod_{j=0}^{\delta-1}\left(x-\alpha^{p^j}\right)$$
is the minimal polynomial of $\alpha$ over $\Z_p$.
\end{proof}

\begin{lemma}\label{lem:st}
Let $s$, $t$, and $m$ be positive integers. If $\gcd(p,2m)=1$, then \begin{enumerate}[\indent$(a)$]
\item\label{item:st=} $\zeta_{2m}^s+\zeta_{2m}^{-s}=\zeta_{2m}^t+\zeta_{2m}^{-t}$ if and only if $s\mod{\pm t}{2m}$; and
\item\label{item:st=-} $\zeta_{2m}^s+\zeta_{2m}^{-s}=-\left(\zeta_{2m}^t+\zeta_{2m}^{-t}\right)$ if and only if $s\mod{m\pm t}{2m}$.
\end{enumerate}
\end{lemma}

\begin{proof}
Since
$$\zeta_{2m}^s+\zeta_{2m}^{-s}-\left(\zeta_{2m}^t+\zeta_{2m}^{-t}\right)=\left(\zeta_{2m}^{s-t}-1\right)\left(\zeta_{2m}^t-\zeta_{2m}^{-s}\right)$$ and $$\zeta_{2m}^s+\zeta_{2m}^{-s}+\zeta_{2m}^t+\zeta_{2m}^{-t}=\left(\zeta_{2m}^{s-t}+1\right)\left(\zeta_{2m}^t+\zeta_{2m}^{-s}\right),$$ our conclusion follows.
\end{proof}

Using Lemmas~\ref{lem:galois} and \ref{lem:st}, we establish the following theorem when $p$ is an odd prime.

\begin{theorem}\label{thm:fullchar}
Let $p>2$ and let $m\geq3$ be an integer such that $\gcd(p,m)=1$. Let $u$ be the order of $p$ modulo $2m$, i.e., $u$ is the smallest positive integer such that $p^{u}\mod{1}{2m}$. Further, let $\delta$ be the degree of the minimal polynomial of $\zeta_4\left(\zeta_{2m}^s+\zeta_{2m}^{-s}\right)$ over $\Z_p$, where $\gcd(s,m)=1$. Then
\begin{itemize}
\item $\delta=\frac{u}{2}$ if
\begin{itemize}
\item[\footnotesize\ding{117}] $p\mod{1}{4}$, $u$ is even, and $p^{\frac{u}{2}}\mod{-1}{2m}$;
\item[\footnotesize\ding{117}] $p\mod{3}{4}$, $u\mod{0}{4}$, and $p^{\frac{u}{2}}\mod{-1}{2m}$; or
\item[\footnotesize\ding{117}] $p\mod{3}{4}$, $u\mod{2}{4}$, and $p^{\frac{u}{2}}\mod{m\pm1}{2m}$;
\end{itemize}
\item $\delta=u$ if
\begin{itemize}
\item[\footnotesize\ding{117}] $p\mod{1}{4}$, $u$ is even, and $p^{\frac{u}{2}}\not\mod{-1}{2m}$;
\item[\footnotesize\ding{117}] $p\mod{1}{4}$, $u$ is odd;
\item[\footnotesize\ding{117}] $p\mod{3}{4}$, $u\mod{0}{4}$, and $p^{\frac{u}{2}}\not\mod{-1}{2m}$; or
\item[\footnotesize\ding{117}] $p\mod{3}{4}$, $u\mod{2}{4}$, and $p^{\frac{u}{2}}\not\mod{m\pm1}{2m}$;
\end{itemize}
\item $\delta=2u$ if $p\mod{3}{4}$ and $u$ is odd.
\end{itemize}
\end{theorem}

\begin{proof}
By Lemma~\ref{lem:galois}, $$\zeta_4^{p^\delta}\left(\zeta_{2m}^{sp^\delta}+\zeta_{2m}^{-sp^\delta}\right)=\zeta_4\left(\zeta_{2m}^s+\zeta_{2m}^{-s}\right).$$ Note that $$\zeta_4^{p^\delta}\left(\zeta_{2m}^{sp^\delta}+\zeta_{2m}^{-sp^\delta}\right)=\pm\zeta_4\left(\zeta_{2m}^{sp^\delta}+\zeta_{2m}^{-sp^\delta}\right).$$ By Lemma~\ref{lem:st}, we have
\begin{equation}\label{eqn:p}
\begin{split}
p^\delta\mod{1}{4}&\text{ and }sp^\delta\mod{\pm s}{2m}\text{, or}\\
p^\delta\mod{3}{4}&\text{ and }sp^\delta\mod{m\pm s}{2m}.
\end{split}
\end{equation}

\textit{Case $1$}: $\gcd(s,2m)=1$. The equations in \eqref{eqn:p} simplify to either $p^\delta\mod{1}{4}$ and $p^\delta\mod{\pm 1}{2m}$, or $p^\delta\mod{3}{4}$ and $p^\delta\mod{m\pm 1}{2m}$.

Since $p^{2u}\mod{1}{4}$ and $p^{2u}\mod{1}{2m}$, we have 
$$\zeta_4^{p^{2u}}\left(\zeta_{2m}^{sp^{2u}}+\zeta_{2m}^{-sp^{2u}}\right)=\zeta_4\left(\zeta_{2m}^s+\zeta_{2m}^{-s}\right),$$
thus $\delta\mid2u$. On the other hand, $p^{2\delta}=(p^\delta)^2$ is congruent to either $1$ or $m^2+1$ modulo $2m$. If $p^{2\delta}\mod{1}{2m}$, then $u\mid2\delta$; if $p^{2\delta}\mod{m^2+1}{2m}$, then note that $p^{2\delta}$ is odd and $2m$ is even, so $m^2+1$ is odd. As a result, $m$ is even and $p^{2\delta}\equiv m^2+1\mod{1}{2m}$, so we again have $u\mid2\delta$. Hence, $\delta\in\left\{\frac{u}{2},u,2u\right\}$.

If $p\mod{1}{4}$, then the smallest positive integer $\delta$ satisfying $p^\delta\mod{\pm1}{2m}$ is $\delta=\frac{u}{2}$ if $p^{\frac{u}{2}}\mod{-1}{2m}$, which implicitly implies that $u$ is even; otherwise, $\delta=u$. If $p\mod{3}{4}$ and $u$ is odd, then $p^{u}\mod{3}{4}$ and $p^{u}\mod{1}{2m}$, so $\delta\neq u$. Also, $\frac{u}{2}$ is not an integer, so $\delta=2u$. Finally, we are left with the case when $p\mod{3}{4}$ and $u$ is even. If $u\mod{0}{4}$ and $p^{\frac{u}{2}}\mod{-1}{2m}$ or $u\mod{2}{4}$ and $p^{\frac{u}{2}}\mod{m\pm1}{2m}$, then $\delta=\frac{u}{2}$; otherwise, $\delta=u$.

\textit{Case $2$}: $\gcd(s,2m)\neq1$. Since $\gcd(s,m)=1$, we deduce that $m$ is odd and $\gcd(s,2m)=2$. As a result, $sp^\delta\mod{m\pm s}{2m}$ in \eqref{eqn:p} will never hold since $sp^\delta$ and $2m$ are even while $m\pm s$ is odd. Hence, \eqref{eqn:p} simplifies to $p^\delta\mod{1}{4}$ and $p^\delta\mod{\pm1}{m}$.

Next, we show that $u$ is the order of $p$ modulo $m$. Let $u'$ be the order of $p$ modulo $m$. Since $p^u\mod{1}{2m}$ implies $p^u\mod{1}{m}$, we have $u'\mid u$. On the other hand, $p^{u'}\mod{1}{2}$ and $p^{u'}\mod{1}{m}$ yield $p^{u'}\mod{1}{2m}$ since $\gcd(2,m)=1$, so we have $u\mid u'$.

With the same argument as in Case 1, we have $\delta\mid2u$. On the other hand, $p^{2\delta}=(p^\delta)^2\mod{1}{m}$, so $u\mid2\delta$. Hence, $\delta\in\left\{\frac{u}{2},u,2u\right\}$. If $p\mod{1}{4}$, then the smallest positive integer $\delta$ satisfying $p^\delta\mod{\pm1}{m}$ is $\delta=\frac{u}{2}$ if $p^{\frac{u}{2}}\mod{-1}{m}$, which implicitly implies that $u$ is even; together with $p^{\frac{u}{2}}\mod{-1}{2}$, we have $p^{\frac{u}{2}}\mod{-1}{2m}$. Otherwise, $\delta=u$. If $p\mod{3}{4}$ and $u$ is odd, then $p^{u}\mod{3}{4}$, so $\delta\neq u$. Also, $\frac{u}{2}$ is not an integer, so $\delta=2u$. Finally, if $p\mod{3}{4}$ and $u$ is even, then the smallest positive integer $\delta$ satisfying $p^\delta\mod{\pm1}{m}$ is $\delta=\frac{u}{2}$ if $p^{\frac{u}{2}}\mod{-1}{m}$ and $\frac{u}{2}$ is even; together with $p^{\frac{u}{2}}\mod{-1}{2}$, we have $p^{\frac{u}{2}}\mod{-1}{2m}$. Otherwise, $\delta=u$.

To complete the proof, we note that $p\mod{3}{4}$, $u\mod{2}{4}$, and $p^{\frac{u}{2}}\mod{m\pm1}{2m}$ will not occur when $m$ is odd. This is because $p^{\frac{u}{2}}$ is odd, while both $m\pm1$ and $2m$ are even.
\end{proof}

To consider the case $p=2$, we first establish a lemma parallel to Lemma~\ref{lem:st}.

\begin{lemma}\label{lem:stp=2}
Let $p=2$, and let $s$, $t$, and $m$ be positive integers. If $\gcd(2,m)=1$, then $\zeta_m^s+\zeta_m^{-s}=\zeta_m^t+\zeta_m^{-t}$ if and only if $s\mod{\pm t}m$.
\end{lemma}

\begin{proof}
Since
$$\zeta_m^s+\zeta_m^{-s}-\left(\zeta_m^t+\zeta_m^{-t}\right)=\left(\zeta_m^{s-t}-1\right)\left(\zeta_m^t-\zeta_m^{-s}\right),$$
our conclusion follows.
\end{proof}

Note that when $p=2$, then $\zeta_4=1$. Furthermore, $\left(\zeta_{2m}^m-1\right)^2=\zeta_{2m}^{2m}-1=0$, so $\zeta_{2m}=\zeta_m$. Hence, $\zeta_4\left(\zeta_{2m}^s+\zeta_{2m}^{-s}\right)=\zeta_m^s+\zeta_m^{-s}$.  With this in mind, we use Lemmas~\ref{lem:galois} and \ref{lem:stp=2} to establish the next theorem.

\begin{theorem}\label{thm:fullchar2}
Let $p=2$ and let $m\geq3$ be an integer such that $\gcd(2,m)=1$. Let $u'$ be the order of $2$ modulo $m$, i.e., $u'$ is the smallest positive integer such that $2^{u'}\mod{1}{m}$. Further, let $\delta$ be the degree of the minimal polynomial of $\zeta_m^s+\zeta_m^{-s}$ over $\Z_2$, where $\gcd(s,m)=1$. Then
$$\delta=\begin{cases}
\frac{u'}{2}&\text{if $u'$ is even and $2^{\frac{u'}{2}}\mod{-1}{m}$};\\
u'&\text{otherwise}.
\end{cases}$$
\end{theorem}

\begin{proof}
By Lemma~\ref{lem:galois},
$$\zeta_m^{s\cdot2^\delta}+\zeta_m^{-s\cdot2^\delta}=\zeta_m^s+\zeta_m^{-s}.$$
By Lemma~\ref{lem:stp=2}, we have
$$s\cdot2^\delta\mod{\pm s}{m},$$
which simplifies to $2^\delta\mod{\pm1}{m}$.

Since $2^{u'}\mod{1}{m}$, we have
$$\zeta_m^{s\cdot2^{u'}}+\zeta_m^{-s\cdot2^{u'}}=\zeta_m^s+\zeta_m^{-s},$$
thus $\delta\mid u'$. On the other hand, $2^{2\delta}=(2^\delta)^2\mod{1}{m}$, so $u'\mid2\delta$. Hence, $\delta\in\{\frac{u'}{2},u'\}$. Therefore, the smallest positive integer $\delta$ satisfying $2^\delta\mod{\pm1}{m}$ is $\delta=\frac{u'}{2}$ if $2^{\frac{u'}{2}}\mod{-1}{m}$, which implicitly implies that $u'$ is even; otherwise, $\delta=u'$.
\end{proof}

Bringing together the results so far presented allows us to establish the main theorem of this section.

\begin{theorem}\label{thm:factorizationform}
Let $m$ be a positive integer such that $\gcd(p,m)=1$, and let $n=p^km$, where $k$ is a nonnegative integer.
\begin{itemize}
\item If $m=1$, then $\Psi_n(x)$ factors in $\Z_p[x]$ as $\left(x^2+4\right)^{\frac{\varphi(p^k)}{2}}$.
\item If $m=2$, then $\Psi_n(x)$ factors in $\Z_p[x]$ as $x^{\varphi(p^k)}$.
\item If $m\geq3$ and $p>2$, then let $\delta$ be defined as in Theorem~$\ref{thm:fullchar}$. In this case, $\Psi_n(x)$ factors in $\Z_p[x]$ as a product of $\varphi(m)/\delta$ distinct irreducible monic polynomials of degree $\delta$, each raised to the $\varphi(p^k)$-th power.
\item If $m\geq3$ and $p=2$, then let $\delta$ be defined as in Theorem~$\ref{thm:fullchar2}$. In this case, $\Psi_n(x)$ factors in $\Z_p[x]$ as a product of $\varphi(m)/(2\delta)$ distinct irreducible monic polynomials of degree $\delta$, each raised to the $2\varphi(p^k)$-th power.
\end{itemize}
\end{theorem}

\begin{proof}
The cases when $m=1$ and $m=2$ follow directly from Theorems~\ref{thm:Psi_p^k} and \ref{thm:Psi_p^km}, respectively. When $m\geq3$ and $p>2$, all roots $\zeta_4\left(\zeta_{2m}^s+\zeta_{2m}^{-s}\right)$ of $\Psi_m(x)$ are distinct by Lemma~\ref{lem:st}, so the result follows from Theorems~\ref{thm:Psi_p^km} and \ref{thm:fullchar}. Finally, when $m\geq3$ and $p=2$, two roots $\zeta_m^s+\zeta_m^{-s}$ and $\zeta_m^t+\zeta_m^{-t}$ of $\Psi_m(x)$ are equal if and only if $t=m-s$ by Lemma~\ref{lem:stp=2}, so the result follows from Theorems~\ref{thm:Psi_p^km} and \ref{thm:fullchar2}.
\end{proof}

\section{Acknowledgments}
These results are based on work supported by the National Science Foundation under the grant numbered DMS-1852378.

\end{document}